\documentclass[12pt,a4paper]{amsart}

\usepackage{amsfonts, amsmath, amssymb, amsthm, amscd, hyperref}

\usepackage{anysize}

\newtheorem{thm}{Theorem}[section]
\newtheorem*{thm*}{Theorem}
\newtheorem{lem}[thm]{Lemma}

\newtheorem{cor}[thm]{Corollary}

\theoremstyle{definition}
\newtheorem{defn}[thm]{Definition}

\theoremstyle{remark}
\newtheorem{rem}[thm]{Remark}

\numberwithin{equation}{section}

\newcommand{\pt}{\mathrm{pt}}

\newcommand{\Sy}{\Sigma^{(2)}}
\newcommand{\Sg}{\Sigma}
\newcommand{\barl}{\bar\ell}
\newcommand{\rH}{\Xi}
\newcommand{\Qo}{{\mathop{Q}\limits^\circ}}
\newcommand{\Po}{{\mathop{P}\limits^\circ}}

\DeclareMathOperator{\Sqe}{Sq_e}

\DeclareMathOperator{\dist}{dist}

\renewcommand{\epsilon}{\varepsilon}

\begin{document}

\title[Regular embeddings of manifolds \dots]{Regular embeddings of manifolds and topology of configuration spaces}

\author{R.N.~Karasev}
\thanks{This research is supported by the Dynasty Foundation, the President's of Russian Federation grant MK-113.2010.1, the Russian Foundation for Basic Research grants 10-01-00096 and 10-01-00139, the Federal Program ``Scientific and scientific-pedagogical staff of innovative Russia'' 2009--2013}

\email{r\_n\_karasev@mail.ru}
\address{
Roman Karasev, Dept. of Mathematics, Moscow Institute of Physics
and Technology, Institutskiy per. 9, Dolgoprudny, Russia 141700}



\begin{abstract}
For a topological space $X$ we study continuous maps $f : X\to \mathbb R^m$ such that images of every pairwise distinct $k$ points are affinely (linearly) independent. Such maps are called affinely (linearly) $k$-regular embeddings. 

We investigate the cohomology obstructions to existence of regular embeddings and give some new lower bounds on the dimension $m$ as function of $X$ and $k$, for the cases $X$ is $\mathbb R^n$ or $X$ is an $n$-dimensional manifold. In the latter case, some nonzero Stiefel--Whitney classes of $X$ help to improve the bound.
\end{abstract}

\maketitle

\section{Introduction}

Let $X$ be a topological space. We study continuous maps $f : X\to \mathbb R^m$ such that images of every pairwise distinct $k$ points are affinely (linearly) independent. We call such maps \emph{affinely (linearly) $k$-regular embeddings}. This concept was introduced in~\cite{bors1957}, it is closely related to some questions of approximation by a system of functions (the Chebyshev approximation), see~\cite{coha1978} for detailed explanations.

For a given space $X$ there are some lower bounds on the dimension $m$, and there are also some existence theorems for large enough $m$. In this paper we consider lower bounds for the cases $X=\mathbb R^n$ or $X$ is some other $n$-dimensional manifold. Obviously, the lower bounds for the case of $\mathbb R^n$ give lower bounds for any $n$-manifolds. In~\cite{borysha1960} the first lower bound from the dimension considerations (nevertheless, nontrivial) was made:

\begin{thm*}[Boltyanskii--Ryshkov--Shashkin, 1960]
If there is an affinely $k$-regular map from $\mathbb R^n$ to $\mathbb R^m$ then
$$
m \ge \left\lfloor\frac{k}{2}\right\rfloor n + \left\lfloor\frac{k-1}{2}\right\rfloor.
$$
\end{thm*}

Then some essentially topological methods were applied to improve this lower bound. In the paper~\cite{chis1979} (the case of $n=2$ was considered previously in~\cite{coha1978}) this bound was improved in the case $n=2^l$. 

\begin{defn}
Denote by $\alpha_p(n)$ the sum of digits in $p$-adic representation of an integer $n$.
\end{defn}

\begin{thm*}[Chisholm, 1979]
If there is a linearly $k$-regular map from $\mathbb R^n$ to $\mathbb R^m$, where $n$ is a power of two, then
$$
m \ge n(k - \alpha_2(k)) + \alpha_2(k).
$$
\end{thm*}
\begin{rem}
This theorem was also rediscovered in~\cite{vas1992,vas1994}.
\end{rem}

The Chisholm theorem was established by considering some equivariant maps, and this method does not distinguish between the affine and linear cases (see Lemma~\ref{inv-bundle} below). From here on we state the results for linearly $k$-regular maps, noting that for affinely $k$-regular maps the lower bound is $1$ less. The Chisholm theorem gives a good estimate for the growth of $m$ as a function of $k$, since every $n$-manifold can be linearly $k$-regularly mapped to $\mathbb R^m$ with $m=(n+1)k+1$ (see~\cite{ha1996,vas1994} for the explanation of the upper bound, in fact every ``general position'' map in such dimension is $k$-regular). 

For the case of manifolds other than $\mathbb R^n$ there is a result from~\cite{ha1996}, using characteristic classes of the manifold:

\begin{thm*}[Handel, 1996]
Suppose $M$ is an $n$-dimensional manifold, $k$ is even, and suppose that the $d$-th dual Stiefel--Whitney class of $M$ is nonzero. If there is a $k$-regular map of $M$ to $\mathbb R^m$ then 
$$
m \ge \frac{k}{2}(n+d+1).
$$
Moreover, if $M$ is compact then
$$
m \ge \frac{k}{2}(n+d+1) + 1.
$$
\end{thm*}

The proof of this theorem in~\cite{ha1996} was incorrect, the map between the configuration spaces (the third formula from the page bottom in~\cite[page~1611]{ha1996}) was defined incorrectly. Informally, $k$ pairwise distinct pairs of points do not necessarily constitute $2k$ pairwise distinct points. Still, in this paper this theorem is rehabilitated and a slightly stronger result (Theorem~\ref{handel2} in Section~\ref{k-partition-sec}) is proved.

We start from an observation (see Lemma~\ref{k-partition}) that it is important to decompose $k$ into a sum of powers of two. It was already known (and obvious) for $M=\mathbb R^n$, but it also works for arbitrary manifold $M$, if we define the appropriate subspaces of the configuration space (see Section~\ref{Qq-defn-sec}). The power of two sum may be the standard binary expansion, as in the theorem of Chisholm, or it may be the sum of $2$'s, as in the theorem of Handel, or something between these two cases.

Using the above remark, we concentrate on the case $k$ a power of two, and denote it by $q$ in this case. In Sections~\ref{ext-sq-sec} and \ref{Sq-coh-sec} the external Steenrod square construction is used to describe the cohomology mod $2$ of the Sylow subgroups of the symmetric group and corresponding configuration subspaces. The results of these sections were previously obtained in~\cite{hung1990}, but we give a self-contained explanation of them.

In Sections~\ref{Qq-coh-sec}, \ref{Rn-emb-sec}, and \ref{Qq-bundle-sec} we give some explicit formulas that allow us to calculate the lower bounds on $m$ for given $M$. In the case $M=\mathbb R^n$ the formulas are almost explicit (see Theorem~\ref{Qq-Rn}). For arbitrary manifolds and $q\ge 4$ there is no general explicit formula (except for the case $q=2$ in the theorem of Handel), but the problem is reduced to some straightforward algebraic calculations with the cohomology of configuration spaces and the Stiefel--Whitney classes of $M$.

In Section~\ref{dim-plus-barl-sec} some particular manifolds $M$ (with a restriction on the dimension and the dual Stiefel--Whitney classes) are considered, and explicit bounds for the dimension of a regular embedding are given. In particular, some products of a projective space with a circle can be taken as $M$.

In Section~\ref{coinc} we apply the computations in the cohomology of the configuration spaces to another problem. We prove the existence of multiple points of continuous maps of a projective space to a Euclidean space, generalizing previous results of the author~\cite{kar2010}.

\vskip 0.5cm
{\bf Acknowledgments.}
The author thanks Peter Landweber and Pavle Bla\-go\-je\-vi\'c for the discussions and useful remarks.

\section{Configuration spaces}

In order to study the images of $k$-tuples of points under some continuous map $f : M\to\mathbb R^n$, it is natural to introduce the configuration space:

\begin{defn} For a topological space $X$ define the \emph{configuration space} by
$$
F_k(X) = \{(x_1, \ldots, x_k)\in X^{\times k} : x_i\neq x_j\ \text{if}\ i\neq j\}.
$$
\end{defn}

Note that the permutation group $\Sg_k$ acts freely on $F_k(X)$.

Denote by $V_k$ the natural $k$-dimensional representation of $\Sg_k$ by permuting the basis vectors. This representation induces a $\Sg_k$-equivariant vector bundle 
$$
V_k\times Y \to Y
$$
over any space $Y$ with action of $\Sg_k$ ($\Sg_k$-space). 

\begin{defn}
For an (equivariant) vector bundle $\xi : E(\xi)\to X$, denote by $\xi^\perp$ an (equivariant) vector bundle of minimal dimension, such that the bundle $\xi\oplus\xi^\perp$ is trivial. This bundle need not be uniquely determined up to isomorphism, but its stable isomorphism class is determined uniquely.
\end{defn}

The main tool in proving the lower bounds for the dimension of $k$-regular maps is the following lemma from~\cite{coha1978}:

\begin{lem}
\label{inv-bundle}
Consider the equivariant bundle $\nu_k(X) : V_k\times F_k(X)\to F_k(X)$. If $\dim\nu_k(X)^\perp = l$ then there is no linearly $k$-regular map $X\to \mathbb R^{k+l-1}$, and no affinely $k$-regular map $X\to \mathbb R^{k+l-2}$.
\end{lem}

It follows from this lemma that we have to study the Stiefel--Whitney (or Pontryagin) classes of $\nu_k(X)^\perp$ and prove that these classes are nonzero in $H_{\Sg_k}^l(F_k(X))$ (with coefficients $\mathbb F_2$ or $\mathbb Z$ respectively) for large enough $l$. The characteristic classes of $\nu_k(X)^\perp$ are usually called the dual Stiefel--Whitney (or Pontryagin) classes of $\nu_k(X)$ and denoted by $\bar w(\nu_k(X))$ and $\bar p(\nu_k(X))$ respectively.

\section{Special configuration subspaces}
\label{Qq-defn-sec}

Let us define a subspace $Q_q(M)$ of the configuration space $F_q(M)$; this subspace is a smooth manifold provided $M$ is a smooth manifold. Such subspaces were introduced for $M=\mathbb R^n$ in~\cite{hung1990} and proved to be useful in determining the cohomology of the symmetric group. They were also used in~\cite{kar2010} to establish some theorems on multiple points of continuous maps, see also Section~\ref{coinc}. We start with the case $M=\mathbb R^n$.

\begin{defn}
Let $q=2^l$ and $\delta>0$. Let $\Qo_1(\mathbb R^n)$ be the configuration, consisting of one point at the origin. 

Let by induction $\Qo_q(\mathbb R^n, \delta)$ be the set of all $q$-point configurations, such that the first $q/2$ points form a configuration of $\Qo_{q/2}(\mathbb R^n, \delta/3)$, shifted by a vector $u$ of length $\delta$, and the other $q/2$ points form a configuration of $\Qo_{q/2}(\mathbb R^n, \delta/3)$, shifted by the vector $-u$.
\end{defn}

Note that $\Qo_q(\mathbb R^n) = \tilde M(n, \log_2 q)$ in the notation of~\cite{hung1990}.

\begin{defn}
\label{Qq-defn2}
A configuration in $\Qo_q(\mathbb R^n, \delta)$ can also be described inductively as $x_1,\ldots, x_q\in\mathbb R^n$ such that all the distances $\dist(x_{2i-1}, x_{2i}) = \dfrac{2\delta}{3^{l-1}}$ and the midpoints of $[x_{2i-1}, x_{2i}]$ form a configuration of $\Qo_{q/2}(\mathbb R^n, \delta)$. 
\end{defn}

Note that $\Qo_q(\mathbb R^n)$ is always a product of $q-1$ spheres of dimension $n-1$. We shall omit $\delta$ from the notation since it does not change the diffeomorphism type of $Q_q(\mathbb R^n)$. Then we can naturally define the fiberwise configuration space $Q_q(\xi)$ for any vector bundle $\xi$ as a bundle of corresponding to the union of $\Qo_q(\xi^{-1}(x))$ for all $x\in M$. This is a subspace of the full fiberwise configuration space $F_q(\xi)$, defined in a similar manner.

Note that Definition~\ref{Qq-defn2} (distance and midpoint characterization) can be applied to any Riemannian manifold $M$, if we allow the last center point (configuration $\Qo_1$) to be any $x\in M$.

\begin{defn}
Let $M$ be a Riemannian manifold. Define $Q_q(M, \delta)\subset F_q(M)$ for $q=2^l$ inductively as follows:

1. $Q_1(M) = M$;

2. For $q\ge 2$ let $Q_q(M, \delta)$ be the set of $q$-tuples $x_1,\ldots, x_q\in M$ such that all the distances $\dist(x_{2i-1}, x_{2i}) = \dfrac{2\delta}{3^{l-1}}$ and the midpoints of $[x_{2i-1}, x_{2i}]$ form a configuration of $Q_{q/2}(M, \delta)$. 
\end{defn}

The following lemma describes $Q_q(M,\delta)$ as a bundle over $M$.

\begin{lem}
\label{metric-Qq}
Let the injectivity radius of $M$ be $r$ and $2\delta < r$. Then $Q_q(M, \delta)$ is a fiber bundle (the bundle map is the last stage midpoint) over $M$, and is naturally homeomorphic to $Q_q(\tau M)$
\end{lem}

\begin{proof}
We prove this by induction. For any configuration $(x_1,\ldots, x_q)\in Q_q(M,\delta)$ the midpoints of pairs $[x_1, x_2], [x_2,x_3],\ldots,[x_{q-1}, x_q]$ form a configuration in $Q_{q/2}(M,\delta)$. Since $2\delta_k<r$, then knowing the midpoint of $[x_1,x_2]$, the possible positions of the points $x_1,x_2$ form a sphere. 

So $Q_q(M,\delta)$ is a product-of-spheres bundle over $Q_{q/2}(M,\delta)$. Moreover, these spheres are spheres of the pullbacks of the tangent bundle $\pi_i^*(\tau M)$, where $\pi_i : Q_{q/2}(M,\delta/3)\to M$ is the map, assigning to a configuration its $i$-th point. Note that the maps $\pi_i$ are all homotopic to the centerpoint map $\pi :Q_{q/2}(M,\delta/3)\to M$ (the homotopy can be obtained by deforming a point $x_{2i-1}$ or $x_{2i}$ to the midpoint of $[x_{2i-1}, x_{2i}]$, and then repeating inductively), hence all the vector bundles are equivalent to $\pi^*(\tau M)$. Now the proof is completed by applying the inductive assumption.
\end{proof}

The space $\Qo_q$ (or $Q_q$) is not invariant under the natural action of $\Sg_q$, but it is invariant under the action of its $2$-Sylow subgroup.

\begin{defn}
Let $q=2^k$. Denote by $\Sy_q$ the Sylow subgroup of $\Sg_q$, generated by all permutations of two consecutive blocks $[a2^l + 1, a2^l + 2^{l-1}]$ and $[a2^l + 2^{l-1}+1, (a+1)2^l]$, where $1\le l \le k$ and $0\le a \le 2^{k-l}-1$. 
\end{defn}

Denote by $A_q$ the subspace of the natural $q$-dimensional representation $V_q$ of $\Sg_q$, consisting of the vectors with zero coordinate sum. As in the previous section, $A_q$ induces the equivariant bundle $\alpha_q(X)$ over any $\Sg_q$-space $X$, the group $\Sg_q$ can be changed to $\Sy_q$.  The following lemma is proved in~\cite{kar2010} by a simple geometric reasoning, it also follows from the results in~\cite{hung1990}.

\begin{lem}
\label{Qq-euler}
The manifold $\Qo_q(\mathbb R^n)$ is $\Sy_q$-invariant. The cohomology $H_{\Sy_q}^{(q-1)(n-1)}( \Qo_q(\mathbb R^n); \mathbb F_2)$ is generated by the Euler class (the topmost Stiefel--Whitney class)
$$
e\left(\alpha_q(\Qo_q(\mathbb R^n))\right)^{n-1}.
$$
\end{lem}

In is well known~\cite{admi2004} that the $\Sg_q$-equivariant cohomology with coefficients $\mathbb F_2$ is mapped injectively to the $\Sy_q$-equivariant cohomology; so we do not lose anything. Actually, we could consider arbitrary $q$, not necessarily a power of two, and define the corresponding subspace $Q_q(M, \delta)$ inductively. It is again invariant under the $2$-Sylow subgroup of $\Sg_q$, but its topmost cohomology is not generated by a power of $e(A_q)$, since the latter class is zero already in $H^*(\Sg_q; \mathbb F_2)$.

\section{Generalization of Lemma~\ref{inv-bundle}}
\label{k-partition-sec}

We are going to generalize Lemma~\ref{inv-bundle}, in order to prove the strengthening of the theorem of Handel~\cite{ha1996} and some more results.

Let us introduce some notation, needed to state the generalizations of Lemma~\ref{inv-bundle}.

\begin{defn} For an (equivariant) vector bundle $\xi : E(\xi)\to X$ denote by $\barl(\xi)$ the maximum $k$ such that the dual (equivariant) Stiefel--Whitney class $\bar w_k(\xi)$ is nonzero.
\end{defn}

It follows from the K\"unneth formula and the multiplicativity of the Stiefel--Whitney classes that for the $\times$-product of vector bundles we have
$$
\barl(\xi\times \zeta) = \barl(\xi) + \barl(\zeta).
$$

Now we are going to state the lemma. It is stated for linearly $k$-regular maps, for affinely $k$-regular maps the lower bound is less by $1$.

\begin{lem}
\label{k-partition} 
Let $k=q_1+\dots+q_l$, where $q_i$ are powers of two. Let $M$ be a smooth manifold. If there exists a linearly $k$-regular map $f:M\to \mathbb R^m$, then
$$
m\ge k - l + 1 + \barl \left( \prod_{i=1}^l \alpha_{q_i}( Q_{q_i}(M)) \right) = k - l + 1 + \sum_{i=1}^l \barl \left(\alpha_{q_i}( Q_{q_i}(M))\right).
$$
\end{lem}

We postpone the proof of Lemma~\ref{k-partition} till the next section. Now let us discuss its consequences. If we apply this lemma to the case $M=\mathbb R^n$, $n$ is a power of two, $k=q_1+\ldots+q_l$ is the binary expansion, then we obtain a slightly weaker result than the Chisholm theorem with the inequality
$$
m\ge n(k-\alpha_2(k)) + 1.
$$
It follows from the fact that $\mathbb R^n$ contains any number of copies of $\mathbb R^n$, and the configuration space $F_k(\mathbb R^n)$ contains the product $\prod_{i=1}^l F_{q_i}(\mathbb R^n)$. In other words, in the case $M=\mathbb R^n$ Lemma~\ref{k-partition} can be modified as follows.

\begin{lem}
\label{k-partition-Rn} 
Let $k=q_1+\dots+q_l$, where $q_i$ are powers of two. Let $n$ be an integer. If there exists a linearly $k$-regular map $f:\mathbb R^n\to \mathbb R^m$, then
$$
m\ge k + \sum_{i=1}^l \barl \left(\alpha_{q_i}( Q_{q_i}(\mathbb R^n))\right).
$$
\end{lem}

Let us give another application of Lemma~\ref{k-partition}, which is a stronger version of the theorem of Handel~\cite{ha1996}.

\begin{thm}
\label{handel2}
Suppose $M$ is an $n$-dimensional manifold, $k$ is an even number. If there is a $k$-regular map of $M$ to $\mathbb R^m$ then 
$$
m \ge \frac{k}{2}(n+\barl(\tau M)+1) + 1.
$$
Moreover, if $M$ is compact then
$$
m \ge \frac{k}{2}(n+\barl(\tau M)+1) + 2.
$$
\end{thm}

\begin{proof}
In~\cite{cf1960} it is shown that 
$$
\barl(\alpha_2(Q_2(M))) = n + \barl (\tau M),
$$
and the case of non-compact $M$ follows from Lemma~\ref{k-partition}. For compact $M$ (see also~\cite{ha1996}) it is possible to replace the last $Q_2(M)$ (see the proof of Lemma~\ref{k-partition} below), which is the space of pairs in $M$ with distance $\delta$, by the space $R^2(M)$, which is the space of pairs in $M$ with distance $\ge \delta$. In~\cite{wu1958,mcc1978} it is shown that 
$$
\barl(\alpha_2(R^2(M))) = \dim M + \barl (\tau M) + 1,
$$
and the estimate on $m$ increases by $1$.
\end{proof}

\section{Proof of Lemma~\ref{k-partition}} 

If the manifold $M$ is compact, then $Q_q(M)$ is a compact manifold again. If $M$ is not compact then we use the following convention. We assume that there exists a compact subset $C\subset M$ (it can be chosen to be a compact manifold with boundary) such that the cohomology map $H^*(M; \mathbb F_2)\to H^*(C; \mathbb F_2)$ is injective, at least on some given finite-dimensional subspace of $H^*(M; \mathbb F_2)$. If we have to make $Q_q(M)$ compact, we consider it as a bundle over $M$, and restrict it to a bundle over $C$. 

We also suppose $M$ to have some Riemannian metric.

\begin{defn}
Denote by $d(Q_q(M,\delta))$ and $D(Q_q(M, \delta))$ the minimum and the maximum distance between some pair of points in a configuration from $Q_q(M,\delta)$. They exist and they are positive under the compactness assumptions above, and they depend continuously on $\delta$.
\end{defn}

We have some freedom to choose $\delta_i$'s in the definitions of $Q_{q_i}(M, \delta_i)$. We are going to choose them in such a way that for each $i=2,\ldots, l$
\begin{equation}
\label{dist-incr}
d(Q_{q_i}(M, \delta_i)) > \sum_{j=1}^{i-1} q_i D(Q_{q_j}(M, \delta_j)),
\end{equation}
and so that the last $\delta_l$ is less than the injectivity radius of $M$. From the continuous dependance of $d(Q_{q_i}(M, \delta_i))$ and $D(Q_{q_i}(M, \delta_i))$ it is possible to satisfy these inequalities if the first $\delta_1$ is chosen small enough. 

Denote by $G=\Sy_{q_1}\times \dots \times \Sy_{q_l}$ the natural symmetry group of $Q_{q_1}(M)\times\dots\times Q_{q_l}(M)$. Now Lemma~\ref{k-partition} is deduced from the following.

\begin{lem}
Under the above assumptions there exists a fiberwise $G$-equivariant map of vector bundles
$$
g: \mathbb R\times \prod_{i=1}^l \alpha_{q_i}(Q_{q_i}(M,\delta_i))\to \mathbb R^m,
$$
where $G$ acts trivially on $\mathbb R$ and $\mathbb R^m$, these two spaces are considered as bundles over one point.
\end{lem}

\begin{proof}
Remind that the bundles $\alpha_{q_i}(Q_{q_i}(M,\delta_i))$ are simply the products $A_{q_i}\times Q_{q_i}(M, \delta_i)$. More precisely, an element of $A_{q_i}\times Q_{q_i}(M, \delta_i)$ is a $q_i$-tuple of points $(x^i_1,\ldots, x^i_{q_i})$ in $M$ along with a set of real coefficients $w^i_1, \ldots, w^i_{q_i}$ with zero sum. The action of $\Sy_{q_i}$ is given by the permutation of the points, and the corresponding permutation of the coefficients.

The required map $g$ of bundles is defined as follows. Suppose we have $l$ sets of $q_i$ ($i=1,\ldots, l$) points each, let the points $x^i_j$ be as above. Suppose we also have the respective coefficients $w^i_j$, and another coefficient $t$. Denote the map $g$ on this combination by
$$
g(\ldots) = t\sum_{i, j} f(x^i_j) + \sum_{i,j} w^i_j f(x^i_j).
$$
This map is obviously $G$-equivariant (if the points are permuted, the coefficients are permuted accordingly), so it is left to show its injectivity on fibers.

The fiber of the vector bundle 
$$
\eta : \mathbb R\times \prod_{i=1}^l \alpha_{q_i}(Q_{q_i}(M,\delta_i))\to \prod_{i=1}^l Q_{q_i}(M,\delta_i)
$$
over the point set $x^i_j$ has the following base: the first vector $e_0$ is given by 
$t=1, w^i_j=0$, then for a given $i$ and $j=1,\ldots, q_i-1$ we have a vector $e^i_j$ with coordinates $w^i_j=1$, $w^i_{j+1}=-1$, the other coordinates being zero. Let us show that the images of the system $\{e_0\}\cup\{e^i_j\}_{i=1,\ldots,l,\ j=1,\ldots,q_i-1}$ are linearly independent.

Suppose that the points $x^i_j$ constitute the point set $V\subset M$, note that the points $x^i_j$ may coincide for different index pairs $(i,j)$. Define the graph $T$ on vertices $V$ as follows: the images of all the pairs $(x^i_j, x^i_{j+1})$ form an edge. 

We claim that $T$ is a tree. Indeed, suppose $T$ has a simple cycle $C\subset V$, consider the maximum index $i$ appearing in an edge $(u,v) = (x^i_j, x^i_{j+1})$ of this cycle. Let $(w, y)$ be the next edge of $C$ with the same index $i$, it may happen that $(w,y) = (u,v)$. The segment of $C$ between $v$ and $w$ goes by the edges $(x^m_j, x^m_{j+1})$ with $m<i$, hence its total length is at most 
$$
\sum_{j=1}^{i-1} q_i D(Q_{q_j}(M, \delta_j)),
$$
but the points $v$ and $w$ are different points in some configuration of $Q_{q_i}(M, \delta_i)$, and the distance between them is at least $d(Q_{q_i}(M, \delta_i))$, which contradicts (\ref{dist-incr}). Thus $T$ is a tree.

Now we see that the images $g(e^i_j)$, written in the basis $f(V)$ (it is a basis since $f$ if linearly $k$-regular), have nonzero coordinate pairs that form a tree $T$. The tree can be reconstructed by adding one new edge and one new vertex at a time, hence the coordinates of these vectors form an upper triangular matrix with nonzero diagonal, after such a reordering of vertices and vectors (edges). In follows that $e^i_j$ are linearly independent. The vector $e_0$ is orthogonal to all of them (if the scalar product is Euclidean in the basis $f(V)$), hence it is independent of the other $e^i_j$.
\end{proof}

\section{External Steenrod squares}
\label{ext-sq-sec}

In order to describe the $\Sy_q$-equivariant cohomology of $\Qo_q(\mathbb R^n)$ and the similar spaces, we have to use the construction of external Steenrod squares. We mostly follow~\cite[Ch.~V]{brs1976}, where the Steenrod squares were defined in the unoriented cobordism. The cobordism was defined using mock bundles, if we allow the mock bundles to have codimension $2$ singularities, we obtain ordinary cohomology modulo $2$. In the sequel we consider the cohomology mod $2$ and omit the coefficients from notation. The similar construction was used in~\cite{hung1990} to calculate the cohomology of $\Qo_q(\mathbb R^n)$, based on the Steenrod decomposition theorem for the cohomology of $(K\times K\times S^n)/\mathbb Z_2$ instead of mock bundles.

The construction of the external Steenrod squares on a polyhedron $K$ starts with the fiber bundle (for some integer $n>0$)
$$
\sigma_{K, n}: (K\times K\times S^n)/\mathbb Z_2\to S^n/\mathbb Z_2=\mathbb RP^n.
$$
The group $\mathbb Z_2$ acts by permuting $K\times K$, and antipodally on $S^n$.
Consider a cohomology class $\xi\in H^*(K)$, represented by a mock bundle $\xi : E(\xi)\to K$. Then the mock bundle 
$$
(\xi\times \xi\times S^n)/\mathbb Z_2\to (K\times K\times S^n)/\mathbb Z_2
$$
is the external Steenrod square $\Sqe \xi$. The operation $\Sqe$ is evidently multiplicative, in~\cite[Ch.~V, Proposition~3.3]{brs1976} it is claimed that $\Sqe$ is also additive. We are going to show that it is not true, first we need a definition.

\begin{defn}
The difference $\Sqe(\xi+\eta) - \Sqe\xi - \Sqe\eta$ is represented by the mock bundle
$$
\xi \odot \eta = (\xi\times\eta\times S^n + \eta\times\xi\times S^n)/\mathbb Z_2,
$$
where $\mathbb Z_2$ exchanges the components $\xi\times\eta$ and $\eta\times\xi$. 
\end{defn}

Since the fiber of $\sigma_{K, n}$ is $K\times K$, the restriction of $\xi\odot \eta$ to the fiber is $\xi\times\eta+\eta\times\xi$, which is nonzero if $\eta\not=\xi$ as cohomology classes. Thus the operation $\odot$ is not trivial.

We need a lemma about the $\odot$-multiplication.

\begin{lem}
\label{odot-ann}
Denote by $c$ the hyperplane class in $H^1(\mathbb RP^n)$. Then for any $\xi,\eta\in H^*(K)$ the product 
$$
(\xi\odot\eta)\smile\sigma_{K,n}^*(c) = 0
$$
in $H^*((K\times K\times S^n)/\mathbb Z_2)$.
\end{lem}

\begin{proof}
Consider the mock bundle 
$$
\alpha = \xi\times\eta\times S^{n-1} + \eta\times\xi\times S^{n-1},
$$
which has the natural $\mathbb Z_2$-action, it represents $(\xi\odot\eta)\smile\sigma_{K,n}^*(c)$ after taking the quotient by the $\mathbb Z_2$-action.

Now divide $S^n$ into the upper and the lower half-spheres $H^+$ and $H^-$. Consider the mock bundle (with boundary) 
$$
\beta = \xi\times\eta\times H^+ + \eta\times\xi\times H^-
$$
over $K\times K\times S^n$. The action of $\mathbb Z_2$ on $\beta$ is defined by permuting the summands and the antipodal identification of $H^+$ and $H^-$. Now it is clear that $\alpha$ is the boundary of $\beta$, and $\alpha/\mathbb Z_2$ is the boundary of $\beta/\mathbb Z_2$. Hence it is zero in the cohomology, and the similar statement is true for the unoriented bordism.
\end{proof}

We have to introduce another operation.

\begin{defn}
Let $\xi :E(\xi)\to K$, $\eta : E(\eta)\to K$ be two mock bundles. Let $p_+,p_-$ be the north and the south poles of $S^n$. Denote the mock bundle over $(K\times K\times S^n)/\mathbb Z_2$ by
$$
\iota(\xi\times\eta) = (\xi\times\eta\times\{p_+\} + \eta\times\xi\times\{p_-\})/\mathbb Z_2.
$$
\end{defn}

It is obvious from the definition that we have relation
$$
\iota(\xi\times\eta)\smile \sigma_{K,n}^*(c) = 0,
$$
it is also obvious that 
$$
\iota(\xi\times\xi) = \Sqe\xi\smile \sigma_{K,n}^*(c)^n.
$$
Let us describe the $\smile$-multiplication of the Steenrod squares, $\odot$, and $\iota(\ldots)$ classes. The following formulas are obvious from the definition:
$$ 
(\xi\odot\eta)\smile(\zeta\odot\chi) = (\xi\smile\zeta)\odot(\eta\smile\chi) + (\xi\smile\chi)\odot(\eta\smile\zeta),
$$
$$
(\xi\odot\eta)\smile(\Sqe\zeta) = (\xi\smile\zeta)\odot(\eta\smile\zeta),
$$
$$
(\xi\odot\eta)\smile\iota(\zeta\odot\chi) = \iota((\xi\smile\zeta)\times(\eta\smile\chi)) + \iota((\xi\smile\chi)\times(\eta\smile\zeta)),
$$
$$
\Sqe\xi\smile\Sqe\eta = \Sqe(\xi\smile\eta),
$$
$$
\Sqe\xi\smile \iota(\eta\times\zeta) = \iota((\xi\smile\eta)\times(\xi\smile\zeta)),
$$
$$
\iota(\xi\times\eta)\smile\iota(\zeta\times\chi)=0.
$$
Now we can describe the structure of the cohomology $H^*( (K\times K\times S^n)/\mathbb Z_2)$.

\begin{defn}
Consider a graded $\mathbb F_2$-algebra $A$ with linear basis $v_1,\ldots, v_m$. Denote by $A\odot A$ the subalgebra of $A\otimes A$, invariant w.r.t. $\mathbb Z_2$-action by permuting the factors. The linear base of $A$ is 
$$
\{v_i\otimes v_i\}_{i=1}^n,\ \{v_i\otimes v_j+v_j\otimes v_i\}_{i<j}. 
$$
\end{defn}

\begin{defn}
Consider a graded $\mathbb F_2$-algebra $A$ with linear basis $v_1,\ldots, v_m$. Denote by $\iota(A\otimes A)$ the quotient vector space $A\otimes A/(v_i\otimes v_j + v_j\otimes v_i)$. As $\mathbb F_2$-algebra it has zero multiplication.
\end{defn}

\begin{lem}
\label{ext-sq}
The maps $\Sqe$ and $\odot$ map the algebra $H^*(K)\odot H^*(K)$ to $H^*( (K\times K\times S^n)/\mathbb Z_2)$. The map $\iota$ maps $\iota(H^*(K)\otimes H^*(K))$ to $H^*( (K\times K\times S^n)/\mathbb Z_2)$. The images of these maps together with the generator $c\in H^1(S^n/\mathbb Z_2)$ multiplicatively generate the cohomology $H^*( (K\times K\times S^n)/\mathbb Z_2)$.

The latter cohomology can be described as the quotient of $H^*(K)\odot H^*(K)\otimes \mathbb F_2[c]\oplus \iota(H^*(K)\otimes H^*(K))$ by the relations
$$
c^{n+1}=0,\ (\xi\odot\eta)\otimes c = 0, \Sqe\xi\otimes c^n = \iota(\xi\otimes\xi).
$$
\end{lem}

Compare this lemma with~\cite[Theorem~2.1]{hung1990}, see also~\cite{step1962}. Note the important particular case: if $n\to\infty$, we image of $\iota(\ldots)$ disappears, and we also can take the quotient of $H^*(K)\odot H^*(K)$ by the linear span of all $\xi\odot\eta$ for $\xi,\eta\in H^*(K)$. Hence, the cohomology $H^*((K\times K\times S^\infty)/\mathbb Z_2)$ has a quotient isomorphic to $\Sqe(H^*(K))\otimes \mathbb F_2[c]$. Here $\Sqe(H^*(K))$ is the same algebra as $H^*(K)$, but with twice larger degrees.  

\begin{proof}
The Leray--Serre spectral sequence for $\sigma_{K, n}$ starts with 
$$
E_2^{p,q} = H^p(\mathbb RP^n; H^q(K\times K)),
$$
where $\mathbb Z_2=\pi_1(\mathbb RP^n)$ permutes the factors of $H^q(K\times K)=H^q(K)\otimes H^q(K)$. Let us decompose the coefficient sheaf $H^*(K\times K)$. If  $v_1,\ldots, v_m$ is the linear basis of $H^*(K)$, then an element $v_i\otimes v_i$ gives a subsheaf, isomorphic to the constant sheaf $\mathbb F_2$. The two elements $v_i\otimes v_j$ and $v_j\otimes v_i$ generate a non-constant sheaf $\mathcal A = \mathbb F_2\oplus \mathbb F_2$ with permutation action of $\pi_1(\mathbb RP^n)$. The cohomology $H^*(\mathbb RP^n; \mathcal A) = H^*(S^n; \mathbb F_2)$, since $\mathcal A$ is the direct image of $\mathbb F_2$ under the natural projection $\pi: S^n\to \mathbb RP^n$. Thus we know the additive structure of $E_2^{*,*}$.

The first column of $E_2$ consists of $\mathbb Z_2$-invariant elements of $H^*(K\times K)$, and all those elements are the restrictions of either $\Sqe\xi$ or $\xi\odot \eta$ to the fiber. Hence all the differentials of the spectral sequence are zero on the first column. The columns between the first and the last ($n$-th) are generated by multiplication with $c$, and the differentials are zero on them too. The last column is isomorphic to $\iota(H^*(K)\otimes H^*(K))$ and the differentials are zero on it from the dimension considerations.

Hence this spectral sequence collapses, that is $E_2=E_\infty$. Let $v_1,\ldots v_m$ be a linear base of $H^*(K)$. The first column of $E_2$ has the linear base 
$$
\{v_i\times v_i\}_{i=1}^m,\ \{v_i\times v_j+ v_j\times v_i\}_{1\le i<j\le m},
$$
the columns $E_2^{j, *}$ with $j=1,2,\ldots,n-1$ have the linear base
$$
\{(v_i\times v_i)c^j\}_{i=1}^m,
$$ 
and the last column has the linear base
$$
\{\iota(v_i\times v_j)\}_{i,j=1}^m.
$$
From the definition of $\Sqe$, $\odot$, and $\iota(\ldots)$ the final cohomology $H^*( (K\times K\times S^n)/\mathbb Z_2)$ is described the same way with $v_i\times v_i$ replaced by $\Sqe v_i$, and $v_i\times v_j+ v_j\times v_i$ replaced by $v_i\odot v_j$. From the relations on $\Sqe$, $\odot$, and $\iota$ it follows that the isomorphism $E_2\cong H^*( (K\times K\times S^n)/\mathbb Z_2)$ is an isomorphism of graded algebras.
\end{proof}

Now consider a vector bundle $\nu : E(\nu)\to K$ and define 
$$
\Sqe\nu : (E(\nu)\times E(\nu)\times S^n)/\mathbb Z_2\to (K\times K\times S^n)/\mathbb Z_2.
$$
The Stiefel--Whitney classes of $\Sqe\nu$ are described by the following lemma.

\begin{lem}
\label{sw-sq}
Let $\dim \nu = k$, and let the Stiefel--Whitney class of $\nu$ be 
$$
w(\nu) = w_0+w_1+\dots + w_k.
$$
Then
$$
w(\Sqe\nu) = \sum_{0\le i < j \le k} w_i\odot w_j + \sum_{i=0}^k (1+c)^{k-i} \Sqe w_i,
$$
where $c$ is the image of the hyperplane class in $H^1(\mathbb RP^n)$.
\end{lem}

\begin{proof}
Consider the case of one-dimensional $\nu$ first. Taking $n$ large enough we do not have to consider the image of $\iota(\ldots)$, then we can return to lesser $n$ by the natural inclusion 
$$
(K\times K\times S^n)/\mathbb Z_2\to (K\times K\times S^{n+m})/\mathbb Z_2. 
$$
The restriction of $\Sqe\nu$ to the fiber $K\times K$ has the Stiefel--Whitney class
$$
w(\nu\times\nu) = 1 + w_1(\nu)\times 1 + 1\times w_1(\nu) + w_1(\nu)\times w_1(\nu).
$$
Hence $w(\Sqe\nu)$ is either $1 + w_1(\nu)\odot 1 + \Sqe w_1(\nu)$, or $1 + w_1(\nu)\odot 1 + c + \Sqe w_1(\nu)$. Any point $x\in K$ gives a natural section
$$
s : S^n/\mathbb Z_2\to (\{x\}\times\{x\}\times S^n)/\mathbb Z_2
$$
of the bundle $\sigma_{K, n}$, and the bundle $s^*(\Sqe\nu)$ over $\mathbb RP^n$ is isomorphic to $\gamma\oplus\epsilon$, where $\gamma$ is the canonical bundle of the projective space, $\epsilon$ is the trivial bundle. Hence we must have 
$$
w(\Sqe\nu) = 1 + w_1(\nu)\odot 1 + c + \Sqe w_1(\nu).
$$

The general formula for $k>1$ follows from the splitting principle, suppose $\nu=\tau_1\oplus\dots\oplus\tau_k$, then
$$
w(\Sqe\nu) = \prod_{i=1}^k (1 + w_1(\tau_i)\odot 1 + c + \Sqe w_1(\tau_i)),
$$
and the result follows by removing parentheses.
\end{proof}

\section{Cohomology mod $2$ of the symmetric group}
\label{Sq-coh-sec}

There are several approaches to the cohomology of the symmetric group, see the books~\cite{admi2004,vas1994}. Here we apply the results of the previous section to describe the cohomology $H^*(\Sy_q; \mathbb F_2)$. This description was obtained by the same method in~\cite{hung1990} but we reproduce it here for completeness.

Consider the groups $\Sy_q$, where $q$ is a power of two. They have an inductive definition as 
$$
\Sy_{2q} = (\Sy_q\times\Sy_q)\rtimes \mathbb Z_2,
$$ 
where the last factor $\mathbb Z_2$ acts by permuting the first two factors $\Sy_q$. This construction is also known as the wreath product 
$$
\Sy_{2q} = \Sy_q\wr \mathbb Z_2.
$$
Hence, the mod $2$ cohomology of $\Sy_{2q}$ can be approximated by the Cartan--Leray spectral sequence (see~\cite{admi2004}) with initial term 
$$
E_2^{p,*} = H^p(\mathbb Z_2; H^*(\Sy_q)\otimes H^*(\Sy_q)),
$$
where $\mathbb Z_2$ acts on $H^*(\Sy_q)\otimes H^*(\Sy_q)$ by permuting the factors. It can be easily seen that this spectral sequence corresponds to the fiber bundle
$$
\begin{CD}
B\Sy_q\times B\Sy_q @>>> B\Sy_{2q}\\
@. @VVV\\
@. B\mathbb Z_2,
\end{CD}
$$
which is the limit case $n\to\infty$ of the external Steenrod square fiber bundle of Section~\ref{ext-sq-sec}. Hence the cohomology $H^*(\Sy_{2q})$ is generated by $H^*(\Sy_q)\odot H^*(\Sy_q)$ and $H^*(\mathbb Z_2)$ with the relations of the form $x\odot y\otimes c = 0$, where $c$ is the generator of $H^1(\mathbb Z_2)$.

We obtain a way to describe the cohomology of $\Sy_q$ by applying repeatedly the external Steenrod square construction. Denote the cohomology algebras of the respective $\mathbb Z_2$ groups in the wreath product $\Sy_q=\mathbb Z_2\wr\dots\wr \mathbb Z_2$ by $\mathbb F_2[c_1],\ldots, \mathbb F_2[c_l]$ (for $q=2^l$). Then we have the inductive formula by Lemma~\ref{ext-sq}:
$$
H^*(\Sy_{2^l}) = \left(H^*(\Sy_{2^{l-1}})\odot H^*(\Sy_{2^{l-1}})\right)\otimes \mathbb F_2[c_l] / (x\odot y\otimes c_l).
$$

The following statement follows from Lemma~\ref{ext-sq} and gives an explicit description of certain quotient algebra of $H^*(\Sy_q)$ (compare with the definition of $\mathcal M^\perp(\ldots)$ in~\cite{hung1990} and~\cite[Proposition~2.8]{hung1990}):

\begin{defn}
Define inductively the ideal $I_q\subset H^*(\Sy_q)$ as generated by the sets
$$
\{x\odot y : x,y\in H^*(\Sy_{q/2})\}\ \text{and}\ \Sqe I_{q/2}.
$$
\end{defn} 

\begin{lem}
\label{reduced-coh}
The algebra $H^*(\Sy_q)/I_q$ (for $q=2^l$) is the polynomial ring 
$$
H^*(\Sy_q)/I_q = \mathbb F_2[\Sqe^{l-1}c_1, \Sqe^{l-2}c_2, \ldots, c_l],
$$
the subalgebra $\mathbb F_2[\Sqe^{l-1}c_1, \Sqe^{l-2}c_2, \ldots, c_l]\subset H^*(\Sy_q)$ is projected onto $H^*(\Sy_q)/I_q$ isomorphically.
\end{lem}

If we consider some $\Sy_q$-space $X$ then the natural equivariant map $\pi_X :X\to \pt$ induces the natural map 
$$
\pi_X^* : H^*(B\Sy_q)=H^*(\Sy_q)\to H_{\Sy_q}^*(X),
$$
thus we speak informally that $\pi_X^*(H^*(\Sy_q))$ is the image of $H^*(\Sy_q)$ in $H_{\Sy_q}^*(X)$. In the sequel we usually consider the subquotient of the cohomology $H_{\Sy_q}^*(X)$, defined as follows:

\begin{defn}
$$
\rH_{\Sy_q}(X) = \pi_X^*(H^*(\Sy_q))/\pi_X^*(I_q).
$$
\end{defn}

Actually, the above reasoning also allows to describe the cohomology of $\Sg_k$ with coefficients $\mathbb F_2$ for $k$ not a power of two. If we consider the binary decomposition $k=q_1+\dots+q_m$, then the Sylow subgroup $\Sy_k = \Sy_{q_1}\times\dots\times\Sy_{q_m}$, and the cohomology algebra is the tensor product of the respective algebras $H^*(\Sy_{q_i})$, described above.

This approach can be applied similarly to the case of cohomology modulo $p$ for odd prime $p$ (compare~\cite[IV.1, Theorem~1.7]{admi2004}). Instead of $\odot$-product we have to use the cyclic product, defined on mock bundles over $K$ as (indexes are modulo $p$)
$$
c(\xi_1, \ldots, \xi_p) = \left(\sum_{i=1}^p \xi_i\times\xi_{i+1}\times\dots\times \xi_{i-1}\right)/\mathbb Z_p.
$$
These cyclic products along with the ordinary external Steenrod $p$-th powers generate the cohomology of $K^{\times p}\times_{\mathbb Z_p} B\mathbb Z_p$. This is obvious at the level of spectral sequences; and it is true on the level of cohomology, since the leftmost column of the spectral sequence survives and multiplicatively generates (along with $H^*(\mathbb Z_p;\mathbb F_p)$) the entire spectral sequence. Then we note that for the $p$-adic decomposition $n=\sum_i p^{k_i}$ we have
$$
\Sigma^{(p)}_n = \prod_i \Sigma^{(p)}_{p^{k_i}}
$$
and
$$
\Sigma^{(p)}_{p^k} = \underbrace{\mathbb Z_p\wr\dots\wr \mathbb Z_p}_k.
$$

\section{Equivariant cohomology of spaces $\Qo_q(\mathbb R^n)$}
\label{Qq-coh-sec}

The results of this section describe the cohomology $H^*(\Qo_q(\mathbb R^n); \mathbb F_2)$ in terms of external Steenrod squares, following mostly~\cite{hung1990}.

The space $\Qo_q(\mathbb R^n)$ is a product of $(n-1)$-dimensional spheres and when $n\to \infty$ we obtain a homotopy trivial space with free $\Sy_q$-action, i.e. a realization of $E\Sy_q$. Denote $\Qo_q(\mathbb R^n)/\Sy_q = \Po_q(\mathbb R^n)$ for brevity, for $q=2$ it is the $(n-1)$-dimensional projective space. It can be easily seen that the inclusion $\Qo_q(\mathbb R^n)\to E\Sy_q$ along with the Steenrod square fibration of the classifying spaces gives a fiber bundle
\begin{equation}
\label{Qq-constr}
\begin{CD}
\Po_q(\mathbb R^n)\times \Po_q(\mathbb R^n) @>>> \Po_{2q}(\mathbb R^n)\\
@. @VVV\\
@. \mathbb RP^{n-1},
\end{CD}
\end{equation}
which is also a particular case of the external Steenrod square fiber bundle. Note that we have the natural cohomology map 
$$
\pi_{\Qo_q(\mathbb R^n)} : H^*(\Sy_q)\to H^*(\Po_q(\mathbb R^n)),
$$
whose image spans a ``large part'' of $H^*(\Po_q(\mathbb R^n))$, but there are also some cohomology classes generated by $\iota(\ldots)$ operation that are not in this image. Note that if we replace $\Qo_q(\mathbb R^n)$ by $F_q(\mathbb R^n)$, then we have the surjectivity for the map $H^*(\Sg_q)\to H^*(F_q(\mathbb R^n)/\Sg_q)$ using the certain cellular structure on $F_q(\mathbb R^n)$, see~\cite{fuks1970,vas1994}, this fact was used in~\cite{kar2010}, but we do not use this fact in this paper. Another interesting fact (not used here) is that the natural restriction $H^*(F_q(\mathbb R^n)/\Sg_q)\to H^*(\Po_q(\mathbb R^n))$ is injective, see~\cite[Theorem~D]{hung1990}.

Still we can describe the subquotient of the cohomology algebra.

\begin{lem}
\label{Qq-reduced-coh}
Let $q=2^l$. The subquotient 
$$
\rH_{\Sy_q} (\Qo_q(\mathbb R^n)) = \rH (\Po_q(\mathbb R^n))
$$
is the polynomial ring $\mathbb F_2[\Sqe^{l-1}c_1, \Sqe^{l-2}c_2, \ldots, c_l]$ with relations 
$$
\forall i = 1,\ldots, l,\ (\Sqe^{l-i}c_i)^n = 0,
$$
where $c_1,\ldots,c_l$ are the generators of the respective $H^1(\mathbb Z_2)$ in the representation 
$$
\Sy_q = \underbrace{\mathbb Z_2\wr \dots \wr \mathbb Z_2}_l.
$$
\end{lem}

\begin{proof}
The cohomology $H^*(\Po_q(\mathbb R^n))$ is obtained from $l$ copies of $H^*(\mathbb RP^{n-1})$ by successive external Steenrod square construction. 

Let us use induction and Lemma~\ref{ext-sq}. From the description of the cohomology of the group $\Sy_q$, the cohomology $\rH^*(\Po_q(\mathbb R^n))$ is generated by $\rH^*(\Po_{q/2}(\mathbb R^n))$ with $\Sqe$ and $\odot$ operations, $\iota(\ldots)$ operation is not used.

The relation on $n$-th powers is obvious, since in every $H^*(\mathbb RP^{n-1})$ we have $c_i^n=0$. Let us prove that there are no other relations in $\rH^*(\Po_q(\mathbb R^n))$. Denote 
$$
h_{k_1 k_2\dots k_l} = (\Sqe^{l-1}c_1)^{k_1}(\Sqe^{l-2}c_2)^{k_2}, \ldots, c_l^{k_l}
$$
and assume the contrary
$$
\sum_{0\le k_1,\ldots, k_l\le n-1} c(k_1,\ldots, k_l) h_{k_1 \dots k_l} = x,
$$
where $x$ is an element from the ideal $I_q$. Choose the lexicographically smallest index $(k_1,\ldots, k_l)$ with nonzero $c(k_1, \ldots, k_l)$ and multiply by $h_{n-1-k_1 \dots n-1-k_l}$, from the $n$-th power relations we have 
$$
h_{n-1\dots n-1} = y
$$
for some $y\in I_q$. It may be proved by induction that all elements $y\in I_q$ of dimension $(q-1)(n-1)$ are mapped to zero under the natural map $\pi_{\Qo_q(\mathbb R^n)} : H^*(\Sy_q) \to H^*(\Po_q(\mathbb R^n))$, informally it follows from the fact that the elements $x\odot y$ do not have the largest possible dimension in the cohomology $H^*((K\times K\times S^{n-1})/\mathbb Z_2)$, if $K$ is a manifold. Thus we have obtained a contradiction.
\end{proof}

Now consider the equivariant bundles over $\Qo_q(\mathbb R^n)$. Denote the bundle 
$$
\alpha_q(\Qo_q(\mathbb R^n)) = \Qo_q(\mathbb R^n)\times A_q
$$ 
simply by $\alpha_q$, it is $\Sy_q$-equivariant and can be also considered as a vector bundle over $\Po_q(\mathbb R^n)$, after going to the quotient by $\Sy_q$ action. 

\begin{lem}
\label{aq-split}
Let $q=2^l$. We have the inductive formula $\alpha_{2^l} = \Sqe(\alpha_{2^{l-1}})\oplus \gamma_l$, where $\gamma_l$ is the pullback of the canonical bundle over $\mathbb RP^{n-1}$ under the natural projection $\Po_{2^l}(\mathbb R^n)\to \mathbb RP^{n-1}$. Applying it repeatedly we obtain
$$
\alpha_{2^l} = \bigoplus_{i=1}^l \Sqe^{l-i}\gamma_i,
$$ 
with $\gamma_i$ being the appropriate pullback on the $i$-th stage of squaring.
\end{lem}

\begin{proof}
The representation $A_{2q}$ has a linear summand, consisting of vectors with the first $q$ coordinates equal, and the last $q$ coordinates equal. This summand is induced from the antipodal action of the quotient $\mathbb Z_2 = \Sy_{2q}/(\Sy_q\times\Sy_q)$ on $\mathbb R$.

The rest of $A_{2q}$ is the direct sum of $A_q$ for the first factor $\Sy_q$ and $A_q$ for the second factor $\Sy_q$, the quotient $\mathbb Z_2 = \Sy_{2q}/(\Sy_q\times\Sy_q)$ acting on it by permuting the summands. This construction corresponds to the $\Sqe$ operation for vector bundles of the form $(X\times A_q)/\Sy_q\to X/\Sy_q$.
\end{proof}

It follows from Lemmas~\ref{aq-split} and \ref{sw-sq} that in the above terms (at $i$-th stage $w(\gamma) = 1 + c_i$)
$$
e(\alpha_q) = \Sqe^{l-1}c_1\Sqe^{l-2}c_2\dots c_l
$$
and
$$
e(\alpha_q)^{n-1} = \Sqe^{l-1}c_1^{n-1}\Sqe^{l-2}c_2^{n-1}\dots c_l^{n-1}.
$$
Note that now Lemma~\ref{Qq-euler} follows from these formulas and Lemma~\ref{ext-sq}. We can also describe the full Stiefel--Whitney class of $\alpha_q$, at least modulo the ideal $I_q$. But according to Lemma~\ref{k-partition-Rn}, we have to describe the bundle $\alpha_q^\perp$ and give a formula for its Stiefel--Whitney class. We need a lemma first.

\begin{lem}
\label{Qq-fro}
Let $n$ be a positive integer, and let $N$ be the least power of two such that $N\ge n$. Then the operator $F_N : x\mapsto x^N$ is zero on the reduced cohomology $\tilde H^*(\Po_q(\mathbb R^n))$.
\end{lem}

\begin{proof}
For $q=2$ it is clear that the $N$-th power operator is zero on $\tilde H^*(\mathbb RP^{n-1})$. Then we proceed by induction. Using the fibre bundle~(\ref{Qq-constr}) we see that all the generators of $\tilde H^*(\Po_q(\mathbb R^n))$ (external Steenrod squares, $\odot$-products, and $\tilde H^*(\mathbb RP^{n-1})$) are annihilated by $F_N$. Since $F_N$ is an algebra homomorphism, then all the reduced cohomology is annihilated by $F_N$.
\end{proof}

\begin{lem}
\label{aq-perp-sw}
Let $q=2^l$, $X$ be a $\Sy_q$-space, and let $N$ be the least power of two such that the map $F_N : x\mapsto x^N$ is zero on the image of $H^*(\Sy_q)$ in $H^*(X/\Sy_q)$. Then 
$$
w(\alpha_q^\perp(X)) = w(\alpha_q(X))^{N-1},
$$
and the Stiefel--Whitney class $w(\alpha_q^\perp(X))$ in the subquotient $\rH_{\Sy_q}(X)$ is expressed in terms of the generators of $H^*(\Sy_q)/I_q$ as follows
$$
w(\alpha_q^\perp(X)) = \sum_{k_1,\ldots, k_l\ge 0} c(k_1, \ldots, k_l)h_{k_1 \dots k_l},
$$
where the coefficient $c(k_1,\ldots, k_l)$ is defined by
\begin{equation}
\label{aq-pow-sw}
c(k_1,\ldots, k_l) = \binom{N-1}{k_1}\cdot\prod_{j=2}^l \binom{2^{j-1}(N-1) - k_{j-1} - 2k_{j-2} - \dots - 2^{j-2}k_1}{k_j},
\end{equation}
if the binomial coefficients are not defined, we assume they are zero.
\end{lem}

\begin{proof}
It is enough to calculate $w(\alpha_q)^{N-1}$ for $X=B\Sy_q$.

In this case the formula is obtained by applying Lemma~\ref{sw-sq} repeatedly, starting from the class
$$
w(\alpha_q^\perp) = (1+c_1)^{N-1}.
$$
\end{proof}

Note that in (\ref{aq-pow-sw}) we can substitute any $m$ instead of $N-1$ and obtain the formula for $w(\alpha_q)^m$ over any $\Sy_q$-space $X$. When applying this lemma to the case $X=\Qo_q(\mathbb R^n)$ we choose $N$ to be the least power of two $\ge n$ by Lemma~\ref{Qq-fro}, and impose the natural conditions $k_1,\ldots, k_l\le n-1$.

\section{Regular embeddings of $\mathbb R^n$}
\label{Rn-emb-sec}

Now we are prepared to consider regular embeddings of $\mathbb R^n$. First, consider one of the simplest cases $q=4$.

\begin{defn}
Denote the function
$$
N(x) = \min \{2^l : 2^l \ge x\}.
$$
\end{defn}

Let $\alpha_4 = \alpha_4(\Qo_4(\mathbb R^n))$. Lemma~\ref{aq-perp-sw} shows that 
$$
\bar w(\alpha_4) = \sum_{0\le k_1,k_2\le n-1} c(k_1, k_2) (\Sqe c_1)^{k_1}c_2^{k_2},
$$
modulo the ideal $I_q$ (generated by $c_1\odot 1$ in this case). The coefficients are
\begin{equation}
c(k_1,k_2) = \binom{N-1}{k_1} \binom{2(N-1) - k_1}{k_2} = \binom{2N-1-k_1-1}{k_2},
\end{equation}
where $N=N(n)$.

It is well-known that the binomial coefficients $\binom{x+y}{y}$ are nonzero iff in the binary representation of $x$ and $y$ none of the positions is taken by $1$ in both $x$ and $y$. Call such two numbers \emph{binary disjoint} and write $x\&y=0$. Since $2N-1$ is a large enough string of $1$'s in the binary representation, then $c(k_1,k_2)\not=0$ iff $(k_1+1)\& k_2=0$. Thus we have
\begin{equation}
\label{bar-l-4}
\barl (\alpha_4) = \max\{2k_1+k_2 : 0\le k_1,k_2\le n-1, (k_1+1)\& k_2 = 0\}.
\end{equation}

\begin{defn}
Define 
$$
\nu(x) = \max\{y \in\mathbb Z^+: y\le x,\ x\&y=0\},
$$
note that for any positive integer $x$
$$
x+\nu(x) = N(x+1)-1.
$$
\end{defn}

\begin{thm}
\label{Q4-Rn}
In the cohomology $H^*(\Po_4(\mathbb R^n))$ modulo $I_q$ we have
$$
\barl \left(\alpha_4(\Qo_4(\mathbb R^n))\right) = 2n-2 + \nu(n) = n + N(n+1)-3.
$$
\end{thm}

\begin{proof}
Let us analyze (\ref{bar-l-4}). If $k_1+1\ge k_2$, than we can assume $k_2=\nu(k_1+1)$, then $\barl(\alpha_4)\ge 2k_1 + \nu(k_1+1)$, and the maximum is attained for $k_1=n-1$.

If $k_1+1<k_2$, then we can assume $k_1 = \nu(k_2)-1$, in this case we have an estimate $\barl (\alpha_4)\ge 2\nu(k_2)-2+k_2 = \nu(k_2) + N(k_2+1)-3$, the maximum is $\nu(n-1)+N(n)-3$, which is less than the previous estimate.
\end{proof}

Now we apply Lemma~\ref{k-partition-Rn} and deduce the following.

\begin{cor}
\label{4k-Rn} 
Let $k$ be divisible by $4$. If there exists a linearly $k$-regular map $f:\mathbb R^n\to \mathbb R^m$, then
$$
m\ge k + \frac{k}{4} \left(n+N(n+1)-3\right).
$$
\end{cor}

The ``greedy'' lower bound in Theorem~\ref{Q4-Rn} using Lemma~\ref{aq-perp-sw} and (\ref{aq-pow-sw}) can be reproduced for any $q=2^l$. Let us state the appropriate result. There is no explicit formula in this theorem, but it can be easily computed in any particular case.

\begin{thm}
\label{Qq-Rn}
In the cohomology $H^*(\Po_q(\mathbb R^n))$ we have
$$
\barl \left(\alpha_q(\Qo_q(\mathbb R^n))\right) \ge \sum_{i=1}^l 2^{l-i}k_i,
$$
where $k_i$ are defined recursively as follows:
$$
k_1 = n-1,
$$
and for $i\ge 2$
$$
k_i = \max\{ x \in \mathbb Z^+: x\le n-1,\ x\&(k_{i-1}+1 + 2(k_{i-2}+1) + \dots + 2^{i-2}(k_1+1)) = 0\},
$$
where $\&$ denote the bitwise `and' operation.
\end{thm}

It is not known whether this bound is the best possible that can be obtained from (\ref{aq-pow-sw}). In case $n$ is a power of two this theorem gives $k_i=n-1$, i.e. the Chisholm theorem. Applying Lemma~\ref{k-partition-Rn}, we obtain the following generalization of the Chisholm theorem.

\begin{cor}
\label{qsum-Rn} 
Denote the lower bound in Theorem~\ref{Qq-Rn} by $l(q, n)$. Suppose $k=q_1+\ldots+q_s$ is a partition of $k$ into powers of two (e.g. the binary representation). If there exists a linearly $k$-regular map $f:\mathbb R^n\to \mathbb R^m$, then
$$
m\ge k + \sum_{i=1}^s l(q_i, n).
$$
\end{cor}

\section{Cohomology of bundles $P_q(\xi)$}
\label{Qq-bundle-sec}

Now consider the bundle $Q_q(\xi)\to M$ associated with some vector bundle $\xi : E(\xi)\to M$. Put $P_q(\xi) = Q_q(\xi)/\Sy_q$. We have the following statement about the equivariant cohomology of $P_q(\xi)$.

\begin{lem}
\label{coh-free-module}
Let $q=2^l$. Suppose $\xi\to M$ is an $n$-dimensional vector bundle. The quotient $H^*(P_q(\xi))/(I_q H^*(M))$ has a free $H^*(M)$-submodule, generated by the classes $h_{k_1\dots k_l}$ with $0\le k_1,\ldots, k_l\le n-1$ from $H^*(\Sy_q)$.
\end{lem}

\begin{proof}
Compare the proof with the proof of Lemma~\ref{Qq-reduced-coh}. Suppose we have a nontrivial relation 
\begin{equation}
\label{l-dep}
\sum_{0\le k_1, \ldots, k_l \le n-1} m(k_1, \ldots, k_l) h_{k_1 \dots k_l} = xm,
\end{equation}
where $m(k_1, \ldots, k_l), m\in H^*(M)$ and $x\in I_q$. Note that 
$$
\pi : P_q(\xi)\to M
$$ 
is a bundle of manifolds, and the cohomology map $\pi_! : H^*(P_q(\xi))\to H^*(M)$ of degree $-(q-1)(n-1)$ is defined. Applying this map to (\ref{l-dep}) we obtain
$$
m(n-1,\ldots, n-1) = 0.
$$
Now consider the lexicographically largest index $(k_1,\ldots, k_l)$ with nonzero $m(k_1, \ldots, k_l)$, multiply (\ref{l-dep}) by $h_{n-1-k_1\dots n-1-k_l}$, and then apply $\pi_!$. Using Lemma~\ref{Qq-reduced-coh} we again obtain $m(k_1, \ldots, k_l)=0$.
\end{proof}

Now consider the ($\Sy_q$-equivariant) dual Stiefel--Whitney class of the bundle $\alpha_q(Q_q(\xi))$ in $H^*(P_q(\xi))$, actually we consider it in $\rH^*(P_q(\xi))$. From naturality of this class it is sufficient to consider the universal bundle $\gamma^n\to G^n$ over the infinite Grassmannian of $n$-subspaces. From Lemma~\ref{coh-free-module} we obtain a decomposition modulo $I_q H^*(G^n)$
\begin{equation}
\label{char-classes-defn}
\bar w(\alpha_q(Q_{G^n}^q(\gamma^n))) \ge \sum_{0\le k_1,\ldots, k_l\le n-1} h_{k_1\dots k_l} t_{k_1 \dots k_l}.
\end{equation}

Hence, the following is proved.

\begin{lem}
\label{char-classes}
Equation~(\ref{char-classes-defn}) defines the characteristic classes $t_{k_1 \dots k_l}(\xi)$ of a vector bundle $\xi$, with the following property:
$$
\barl \left(\alpha_q(Q_q(\xi))\right) \ge \max_{0\le k_1,\ldots, k_l\le n-1} \{\dim  h_{k_1\dots k_l} + \dim t_{k_1 \dots k_l}(\xi)\},
$$ 
where by the dimension if a cohomology class $t_{k_1 \dots k_l}(\xi)$ we mean the maximum dimension of a nonzero homogeneous component of $t_{k_1 \dots k_l}(\xi)$.
\end{lem}

The computation may be simpler for the following subset of these characteristic classes:

\begin{defn}
Define the characteristic classes 
$$
T_q(\xi) = t_{n-1 \dots n-1}^q(\xi).
$$
We have
$$
T_q(\xi) = \pi_!(\bar w(\alpha_q(Q_q(\xi)))),
$$
since $h_{n-1\dots n-1}$ is the fundamental class of the fiber manifold $\Po_q(\mathbb R^n)$.
\end{defn}

By Lemma~\ref{char-classes} for the class $T_q(\xi)$ we have 
$$
\barl \left(\alpha_q(Q_q(\xi))\right) \ge (q-1)(n-1) + \dim T_q(\xi),
$$ 
noting that if $T_q(\xi)$ is zero we put $\dim T_q(\xi) = -\infty$.

In the case $q=2$ the class $T_2(\xi)$ is the dual Stiefel--Whitney class of $M$, as it was already used in the proof of Theorem~\ref{handel2}. The calculations are harder even in the case $q=4$, compare the calculation of coincident $q$-tuple characteristic classes in~\cite{kar2010}. Note also that $T_q(\xi)$ (unlike the classes in~\cite{kar2010}) is not stable under summation with a trivial bundle and depends on the dimension of $\xi$.

Now consider a way to calculate the above characteristic classes. Take a power of two $N$ such that the map $x\mapsto x^N$ is zero on $H^*(P_q(\xi))$. In this case the dual Stiefel--Whitney class of $\alpha_q(Q_q(\xi))$ modulo the ideal $I_q$ is given by (\ref{aq-perp-sw}), because we have 
$$
w(\alpha_q(Q_q(\xi)))^N=1.
$$ 
In order to have the decomposition (\ref{char-classes-defn}) we have to express the monomials ($q=2^l$) 
$$
(\Sqe^{l-1} c_1)^{k_1}\dots(\Sqe c_{l-1})^{k_{l-1}}(c_l)^{k_l}
$$ 
with some $k_i\ge n = \dim \xi$ in terms of the similar monomials with all $k_i\le n-1$. This expression is done modulo $I_q H^*(M)$. The following lemma gives the needed relations. The bundles $\Sqe^{l-i}\gamma_i$ were defined in Lemma~\ref{aq-split} over the space $\Qo_q(\mathbb R^n)$, but actually they arise from the corresponding representation of $\Sy_q$, and therefore they are defined over any $\Sy_q$-space.

\begin{lem}
\label{Qq-bundle-mul}
Let $q=2^l$, and $\pi: Q_q(\xi)\to M$ be the natural projection. Then we have the relations for $i=1,\ldots, l$
$$
e(\Sqe^{l-i}\gamma_i\otimes \pi^*(\xi)) = 0
$$
in the cohomology $H^*(P_q(\xi))$.
\end{lem}

\begin{proof}
Consider the natural map $s : Q_q(\xi) \to \xi^{\oplus q}$, it can be considered as an equivariant section of the vector bundle $\pi^*(\xi)^{\oplus q}$ over $Q_q(\xi)$. We have
$$
s(p_1, \ldots, p_q) = p_1\oplus\dots\oplus p_q,
$$ 
and it is readily seen from the definition of $\Qo_q(\mathbb R^n)$, that the sum of $p_i$ is zero. Hence we have a section of $\alpha_q(Q_q(\xi))\otimes \pi^*(\xi)$. By Lemma~\ref{aq-split} (which is true for the representations) we have 
$$
\alpha_q(Q_q(\xi)) = \bigoplus_{i=1}^l \Sqe^{l-i}\gamma_i,
$$
and therefore
$$
\alpha_q(Q_q(\xi))\otimes \pi^*(\xi) = \bigoplus_{i=1}^l \Sqe^{l-i}\gamma_i\otimes \pi^*(\xi).
$$
It is easily seen that $s$ gives a nonzero section for all the summands, after the corresponding projection. Thus their equivariant Euler classes are zero.
\end{proof}

Note that this lemma expresses $(\Sqe^{l-i} c_i)^n$ (modulo $I_q H^*(M)$) through the combinations of $(\Sqe^{l-i} c_i)^{k_i}$ with $k_i<n$ and $(\Sqe^{l-j} c_j)^{k_j}$ with $j>i$ and $k_j$ not necessarily $<n$. Combining (\ref{aq-perp-sw}) and the above lemma, we obtain a way to calculate $\barl \left(\alpha_q(Q_q(\xi))\right)$ (modulo $I_q H^*(M)$) in every particular case.

\section{Some explicit bounds for regular embeddings of manifolds}
\label{dim-plus-barl-sec}

Let us give more explicit examples of lower bounds for regular embeddings of manifolds in some particular cases. 

Consider a vector bundle $\xi :E(\xi)\to M$ and its spaces $Q_q(\xi)$ and $P_q(\xi)$. We need a claim about the nilpotence degree of the classes $\Sqe^{l-i} c_i$ in $\rH^*(P_q(\xi))$. The first lemma is a general statement, the second is its application to the cohomology of $P_q(\xi)$.

\begin{lem}
\label{euler-nil}
Let $\xi: E(\xi)\to X$ and $\eta : E(\eta)\to X$ be two vector bundles over a topological space $X$. If $e(\xi\otimes \eta) = 0$, then 
$$
e(\eta)^{\dim\xi+\barl(\xi)} = 0.
$$
\end{lem}

\begin{lem}
\label{Sqe-nil}
Let $q=2^l$, and let $\xi: E(\xi)\to M$ be a vector bundle over $M$. Then we have the relations for $i=1,\ldots, l$
$$
(\Sqe^{l-i}c^i)^{\dim\xi+\barl(\xi)} = 0
$$
in the cohomology $H^*(P_q(\xi))$.
\end{lem}

Lemma~\ref{Sqe-nil} follows from Lemma~\ref{Qq-bundle-mul} and Lemma~\ref{euler-nil}. So we have to prove Lemma~\ref{euler-nil}.

\begin{proof}[Proof of Lemma~\ref{euler-nil}]
By the splitting principle we can assume that the bundle $\eta$ is a sum of line bundles
$$
\eta = \eta_1\oplus\dots\oplus\eta_m.
$$
Denote the characteristic classes 
$$
e(\eta_i) = y_i,\quad w(\xi) = 1 + w_1 + \dots + w_n,\quad \bar w(\xi) = 1 + \bar w_1 + \dots + \bar w_k.
$$
We have the equation
$$
e(\xi\otimes\eta) = \prod_{i=1}^m (y_i^n + w_1 y_i^{n-1} + \dots + w_n) = 0.
$$
Multiplying by 
$$
\prod_{i=1}^m (y_i^k + \bar w_1 y_i^{k-1} + \dots + \bar w_k)
$$
we obtain 
$$
e(\eta)^{n+k} = \prod_{i=1}^m y_i^{n+k} = 0.
$$
\end{proof}

Now we can calculate $\barl(\alpha_q(Q_q(\xi)))$ in a particular case (compare the Chisholm theorem).

\begin{lem}
\label{dim-plus-barl-pow2}
Suppose $\xi : E(\xi)\to M$ is a vector bundle over $M$. Let $q=2^l$, $\dim\xi = n$, $\barl(\xi)=d$. Suppose that either $n+d$ is a power of two and $\bar w_d(\xi)^{q-1}\not=0$, or $q=2$. Then
$$
\barl\left(\alpha_q(Q_q(\xi))\right) = (q-1)(n+d-1).
$$ 
\end{lem}

\begin{proof}
Denote $\alpha_q = \alpha_q(Q_q(\xi))$. The class $\bar w(\alpha_q)$ modulo the ideal $I_q$ is given by (\ref{aq-perp-sw}), we can take $N=n+d$ in this equation since the map $x\mapsto x^{n+d}$ sends the Stiefel--Whitney classes $w_1(\alpha_q),\ldots,w_{q-1}(\alpha_q)$ to zero (modulo $I_q$). Hence we have to prove that the class (the leading term of (\ref{aq-perp-sw}))
$$
\left(\Sqe^{l-1} c_1\dots\Sqe^{1}c_{l-1}c_l\right)^{n+d-1}
$$
is not zero in $H^*(P_q(\xi))$. From~\cite[Theorem~1]{kar2010} it follows that under the assumption $\barl(\xi) = d$ we have a relation
$$
\left(\Sqe^{l-1} c_1\dots\Sqe^{1}c_{l-1}c_l\right)^{n+d-1} = \left(\Sqe^{l-1} c_1\dots\Sqe^{1}c_{l-1}c_l\right)^{n-1} \bar w_d(\xi)^{q-1}
$$
and the result follows.
\end{proof}

Now an estimate for the dimension of regular embeddings follows from Lemma~\ref{dim-plus-barl-pow2}.

\begin{thm}
\label{dim-plus-barl-emb}
Let $M$ be an $n$-dimensional manifold. Let $\barl(\tau M)=d$, and $n+d$ be a power of two. Suppose that $k=q_1+\ldots + q_l$ is a sum of powers of two and $\bar w_d(\xi)^{q_i-1}\not=0$ for any $i$. 

Under the above assumptions, if there exists a linearly $k$-regular embedding $M\to \mathbb R^m$ then
$$
m\ge k - l + 1 + \sum_{i=1}^l (q_i-1)(n+d-1) = (k-l)(n+d) + 1.
$$
\end{thm}

\begin{proof}
The result follows from Lemmas~\ref{dim-plus-barl-pow2} and \ref{k-partition}.
\end{proof}

To give an explicit application of Theorem~\ref{dim-plus-barl-emb}, consider $M=\mathbb RP^{n-1}\times S^1$, $n=2^p-d$, $p$ and $d$ some positive integers such that $2^p-1-d>0$. The direct calculations (see also Section~\ref{coinc}) show that 
$$
\barl(\tau M) = d,\quad \text{if}\ (q-1)d\le n-1,\ \text{then}\ \bar w_d(\tau M)^{q-1}\not =0.
$$
By Theorem~\ref{dim-plus-barl-emb}, if $k=q_1+\dots+q_l$ is a sum of powers of two, and any $q_i\le \dfrac{n-1}{d}+1$, then the dimension of linearly $k$-regular embedding of $M$ is at least $m\ge (k-l)(n+d)+1$. If the number $k$ is itself at most $\dfrac{n-1}{d}+1$, then we can take its binary representation, in this case $m\ge (k-\alpha_2(k))(n+d) +1$.

\section{Regular embeddings and the tangent bundle of $P_q(M)$}

Let us describe another approach to lower bounds for the dimension of $k$-regular embedding, not using Lemma~\ref{k-partition}. The method of Boltyanskii--Ryshkov--Shashkin actually shows that any affinely (linearly) $2q$-regular embedding $M\to \mathbb R^m$ gives a continuous injective map 
$$
F_q(M)\times D^{q-1} \to \mathbb R^m
$$
in the affine case, or 
$$
F_q(M)\times D^{q-1}\to S^{m-1}
$$
in the linear case. Here $D^l$ is an $l$-dimensional open disc, $q$ is not necessarily a power of two. Then the dimension considerations give either $m\ge n(q+1)-1$ or $m\ge n(q+1)$ respectively.

This reasoning can be improved in some cases. Consider linear embeddings and let $q$ be a power of two. The above map is restricted to an injective map
$$
F_q(M)/{\Sg_q} \to S^{m-1}.
$$
The space $F_q(M)$ contains a submanifold $Q_q(M)$, and we obtain an injective continuous map
$$
P_q(M) = Q_q(M)/\Sy_q \to S^{m-1}.
$$
According to~\cite{wu1958,cf1960}, the existence of such a map implies the inequality ($\tau$ denotes the tangent bundle of a manifold)
$$
m-1 \ge \dim P_q(M) + \barl (\tau P_q(M)),
$$
or (for compact $M$)
$$
m-1 \ge \dim P_q(M) + \barl (\tau P_q(M)) + 1.
$$

We obviously have to describe the tangent bundle of $Q_q(M)$ and the action of $\Sy_q$ on it. Consider the case of $\Qo_q(\mathbb R^n)$ first. Denote $\mathbb R^n=L$ for brevity. One of the descriptions of $\Qo_q(L)$ identifies it with the product of $q-1$ spheres of $L$, hence we  
have an embedding
$$
\Qo_q(L)\subset L^{q-1}.
$$
The tangent vector of $\Qo_q(L)$ at a point $(p_1, \ldots, p_{q-1})\in L^{q-1}$ is a vector $(v_1,\ldots, v_{q-1})\in L^{q-1}$, such that $p_i$ and the respective $v_i$ are orthogonal for any $i$. The action of the generators (block permutations) of $\Sy_q$ is given by reversing one $p_i$ and $v_i$, and permuting some other $p_j$'s and $v_j$'s, according to the binary tree structure. Consider also the $q-1$-dimensional bundle $\eta$ over $\Qo_q(L)$ such that the fiber of $\eta$ over $(p_1, \ldots, p_{q-1})$ is the set of $q-1$-tuples $(u_1,\ldots, u_{q-1})$, such that any $u_i$ is parallel to the respective $p_i$. Let $\Sy_q$ act on $(u_1,\ldots, u_{q-1})$ in the same way, as on $(p_1,\ldots, p_{q-1})$. The numbers $u_i/p_i$ give an $\Sy_q$-invariant identification with the trivial bundle
$$
\eta = \epsilon^{q-1},
$$
and from the obvious identification $L^{q-1}=A_q\otimes L$ we have
$$
\tau\Qo_q(L) \oplus \epsilon^{q-1} = \alpha_q(\Qo_q(L))\otimes L.
$$

For an arbitrary manifold $M$ we similarly obtain ($\pi : Q_q(M)\to M$ is the natural projection)
$$
\tau Q_q(M)\oplus \epsilon^{q-1} = \nu_q(Q_q(M))\otimes \pi^*(\tau M),
$$
since the fiberwise tangent bundle is $\alpha_q(Q_q(M))\otimes \pi^*(\tau M) - \epsilon^{q-1}$, the fiberwise orthogonal bundle is $\pi^*(\tau M)$, and $\nu_q = \alpha_q\oplus\epsilon$ by definition. Thus we have proved the following.

\begin{thm}
\label{brs-emb}
Let $k$ be a power of two, $M$ be an $n$-dimensional manifold. If there exists a linearly $k$-regular map $f:M\to \mathbb R^m$, then
$$
m \ge (n-1)(k/2-1) + \barl \left( \nu_{k/2}(Q_{k/2}(M))\otimes \pi^*(\tau M) \right) + 1,
$$
or (for compact $M$)
$$
m \ge (n-1)(k/2-1) + \barl \left( \nu_{k/2}(Q_{k/2}(M))\otimes \pi^*(\tau M) \right) + 2.
$$
\end{thm}

In the case $M=\mathbb R^n$ this theorem gives a worse estimate, compared to the Chisholm theorem, but for other manifolds this bound can be useful. 

\section{Multiplicity of maps from projective spaces to Euclidean spaces}
\label{coinc}

In~\cite{kar2010} it was shown that continuous maps $f : \mathbb RP^m\to \mathbb R^n$ must have coincident $q$-tuples under certain restrictions on $q$, $m$, $n$. This was proved without any computation in the cohomology of the symmetric group by some geometric reasoning. Using the above description of the cohomology of the space $\Qo_q(\mathbb R^n)$ modulo the ideal $I_q$, it is possible to generalize the result.

\begin{thm}
\label{proj-coinc}
Let $q=2^l$, $n\ge m$ be positive integers. Put $d=n-m$, $p=N(m+1) - m-1$. Suppose for certain $0\le k_1,\ldots, k_l\le p - d - 1$ such that 
$$
(2^l-1)p - k_1 - 2 k_2 - \dots - 2^{l-1} k_l \le m
$$ 
the coefficient
\begin{equation}
\label{aqm-coeff}
c(k_1,\ldots, k_l) = \binom{p}{k_1}\cdot \prod_{j=2}^l \binom{2^{j-1}p - k_{j-1} - 2k_{j-2} - \dots - 2^{j-2}k_1}{k_j},
\end{equation}
is odd. Then any continuous map $f : \mathbb RP^m\to \mathbb R^n$ has a coincident $q$-tuple from $Q_q(\mathbb RP^m)$.
\end{thm}

The result of~\cite{kar2010} follows from this theorem by putting $k_1=k_2=\dots = k_l=0$.

\begin{proof}
Put $M=\mathbb RP^m$ for brevity, and let $\tau M^\perp$ have dimension $p'$. It was shown in~\cite{kar2010} that a coincident $q$-tuple of $f:M\to \mathbb R^n$ from $Q_q(M)$ is guaranteed by the Euler class of the vector bundle $\alpha_q\otimes (\epsilon^n\oplus \tau M^\perp)$ over $\Po_q(\mathbb R^{m+p'})\times M$. It is well-known that the Stiefel--Whitney class of $\tau M^\perp$ is 
$$
w(\tau M^\perp) = (1+c)^p,
$$
where $c$ in the generator of $H^1(M)$ and $p=N(m+1)-m-1\le p'$. Then by Lemma~\ref{aq-perp-sw} (note the remark after it) we have the equation modulo $I_q H^*(M)$
\begin{multline*}
e(\alpha_q\otimes (\tau M^\perp\otimes \epsilon^n) ) = (\Sqe^{l-1} c_1\dots c_l)^{n+p'-p} \cdot\\\cdot\sum_{k_1,\dots, k_l\ge 0} c(k_1, \dots, k_l)(\Sqe^{l-1} c_1)^{k_1}\dots(c_l)^{k_l}\times c^{p(2^l-1) - k_1 - 2k_2 - \dots -2^{l-1}k_l},
\end{multline*}
where the coefficients $c(k_1,\ldots, k_l)$ are as in (\ref{aqm-coeff}). 

Now we note that in the cohomology $H^*(\Po_q(\mathbb R^{m+p'})\times M)$ we have relations 
$$
\forall i=1,\ldots, l\quad (\Sqe^{l-i} c_i)^{m+p'} = 0,\quad c^{m+1} = 0,
$$
that imply the inequalities $k_1,\ldots, k_l\le p - d - 1$ and
$$
(2^l-1)p - k_1 - 2 k_2 - \dots - 2^{l-1} k_l \le m
$$
respectively. Thus the result follows.
\end{proof}

\end{document}